\documentclass[12pt]{amsart}
\usepackage{amssymb, latexsym}

\theoremstyle{plain}
\newtheorem{theorem}{Theorem}

\newtheorem{lemma}[theorem]{Lemma}
\newtheorem{corollary}[theorem]{Corollary}

\theoremstyle{definition}

\theoremstyle{remark}

\newcommand\oo{\infty}
\newcommand\eps{\varepsilon}
\newcommand\fpp{fixed point property }
\newcommand\N{\mathbb{N}}

\newcommand\x{\widetilde{x}}
\newcommand\y{\widetilde{y}}
\newcommand\U{\mathcal{U}}
\newcommand\bfy{\mathbf{y}}
\newcommand\ukk{the weak$^*$ uniform Kadec-Klee property }

\title{The fixed point property via  dual space properties}
\date{\today}

\author{P.N. Dowling}
\address{Department of Mathematics and Statistics, Miami University,
 Oxford, OH 45056}
\email{dowlinpn@muohio.edu}

\author{B. Randrianantoanina}
\address{Department of Mathematics and Statistics, Miami University,
 Oxford, OH 45056} \email{randrib@muohio.edu}

\author{B. Turett}
\address{Department of Mathematics and Statistics, Miami University,
 Oxford, OH 45056}
\curraddr{Department of Mathematics and Statistics, Oakland University,
Rochester, MI 48309} \email{turett@oakland.edu}

\begin{document}

\begin{abstract}
A Banach space has
the weak fixed point property if its dual space has a
weak$^*$ sequentially compact unit ball and the dual space
satisfies the weak$^*$ uniform Kadec-Klee property; and it
has the \fpp if there exists $\varepsilon>0$ such that, for  every
 infinite subset $A$ of the unit sphere of the dual space,
 $A\cup (-A)$ fails to be $(2-\varepsilon)$-separated.  In
particular, $E$-convex Banach spaces, a class of spaces
that includes the uniformly nonsquare spaces, have the fixed
point property.
\end{abstract}

\maketitle


Determining conditions on a Banach space $X$ so that every nonexpansive
mapping from a nonempty, closed, bounded, convex subset of $X$ into itself
has a fixed point has been of considerable interest for many years.  A Banach
space has the fixed point property if, for each nonempty, closed, bounded,
convex subset $C$ of $X$, every  nonexpansive mapping of  $C$
 into itself has a fixed point.  A Banach space is said to have the weak
 \fpp if the class of sets $C$ above is restricted to the set of weakly
 compact convex sets; and a Banach space is said to have the weak$^*$
 \fpp if $X$ is a dual space and the class of sets $C$ is restricted to the set
 of weak$^*$ compact convex subsets of $X$.

A well-known open problem in Banach spaces is whether every reflexive
Banach space has the \fpp for nonexpansive mappings.  The question of
whether more restrictive classes of reflexive spaces, such as the class of
superreflexive Banach spaces or Banach spaces isomorphic to the
Hilbert space $\ell^2$, have the \fpp has also long been open
and has been investigated by many authors \cite{GK, L2, LF, MN}.
Recently, J.\ Garc{\'i}a-Falset, E.\ Llorens-Fuster, and
E.M.\ Mazcu{\~n}an-Navarro \cite{GF}
proved that uniformly nonsquare Banach spaces, a sub-class
of the superreflexive spaces, have the fixed point property.
In this article, it is shown that the larger class of $E$-convex Banach spaces
have the fixed point property.  The $E$-convex Banach spaces,
introduced by S.V.R.\ Naidu and K.P.R.\ Sastry \cite{NS}, are a
class of Banach spaces lying strictly between the uniformly nonsquare
Banach spaces and the superreflexive spaces (see also \cite{AF}).

The second geometric property of Banach spaces that is considered in
this article is \ukk in a dual Banach space.   A dual space $X^*$ has \ukk
if, for every $\eps>0$, there exists $\delta >0$ such that,
if $(x_n^*)$ is a sequence in the unit ball of $X^*$ converging weak$^*$
to $x^*$ and the separation constant
${\text{sep}} (x_n^*) \overset{\text{def}}{=} \inf\{\| x_n^* - x_m^*\| : m\ne n\} > \eps$,
then $\|x^*\| < 1 - \delta$.  It is well-known \cite{vDS} that, if $X^*$ has
weak$^*$ uniform Kadec-Klee property, then
$X^*$ has the weak$^*$ fixed point property.
If, in addition, the unit ball of $X^*$ is weak$^*$ sequentially compact, more
is true:   Theorem \ref{T:w*ukk} notes that,
if $X^*$ has weak$^*$ uniform Kadec-Klee property and the unit ball of $X^*$ is
weak$^*$ sequentially compact, then
 $X$  has the weak fixed point property.  As a consequence
of Theorem \ref{T:w*ukk}, it is noted that several nonreflexive Banach
spaces such as quotients of $c_0$ and $C(T)/A_0$, the predual of $H^1$, have the weak
fixed point property.

Since the proofs of the main theorems in this paper will require elements of the proof
that uniformly nonsquare Banach spaces have the fixed point
property, a complete proof of this known result is presented.
The proof presented here is  a distillation of the original proof
and combines elements from  \cite[Th. 2.2]{DB} and
\cite[Th. 3.3]{GF}.  Recall  that a Banach space $X$ is uniformly
nonsquare  \cite{J} if there exists $\delta>0$ such that, if $x$ and $y$
are in the unit ball of $X$, then either $\| (x+y)/2\| < 1-\delta$
or $\| (x-y)/2\| < 1-\delta$.

The general set-up in proving that a Banach space has the weak \fpp
 has, by now, become standard fare.  If a Banach space $X$
fails to have the weak fixed point property, there exists a nonempty,
weakly compact, convex set $C$ in $X$ and a nonexpansive mapping
$T:C\to C$ without a fixed point.  Since $C$ is weakly compact, it is
possible by Zorn's Lemma to find a minimal $T$-invariant,
weakly compact, convex subset $K$
of $C$ such that $T$ has no fixed point in $K$.  Since the diameter of $K$ is
positive (otherwise $K$ would be a singleton and $T$ would have a
fixed point in $K$), it can be assumed that the diameter of $K$ is $1$.
It is well-known that there exists an approximate fixed point sequence $(x_n)$
for  $T$ in $K$ and, without loss of generality, we may assume that $(x_n)$
converges weakly to 0.  For details on this general set-up, see \cite[Chapter 3]{GK}.

\begin{theorem} [ Garc{\'i}a-Falset, Llorens-Fuster, and
E.M. Mazcu{\~n}an-Navarro \cite{GF}] \label{T:uns}
Uniformly nonsquare Banach spaces have the fixed
  point property for nonexpansive mappings.
\end{theorem}

\begin{proof}
Assume that a Banach space $X$ fails to have the fixed point property.
Since uniformly nonsquare spaces are reflexive \cite{J}, the \fpp and the
weak \fpp coincide for $X$.  Therefore there exists a nonexpansive map
$T:K\to K$ without a fixed point where $K$ is a minimal $T$-invariant set in $X$
with diameter $1$.  Let $(x_n)$ be an approximate fixed point sequence for
 $T$ in $K$ and assume that $(x_n)$
converges weakly to 0.

Consider the set in $\ell^\oo(X)/c_0(X)$ defined by
\[
[W] = \{\, [z_n] \in [K] : \Vert [z_n] - [x_n]\Vert \leq 1/2\ \text{and}\
\limsup_n \limsup_m \Vert z_m - z_n\Vert \leq
1/2\, \}\,.
\]
It is easy to check that $[W]$ is closed, bounded, convex, nonempty (since $[\frac12 x_n]$
is in the set), and $[T]$-invariant where $[T][z_n] \overset{\text{def}}{=} [T(z_n)]$.
So, by a result of Lin \cite{L},
$\displaystyle\sup_{[z_n]\in [W]} \Vert [z_n] -x \Vert = 1$ for each $x\in K$.  In particular, with
$x = 0$,
 $\displaystyle\sup_{[z_n]\in [W]} \Vert [z_n]  \Vert = 1$\,.

Let $\eps>0$ and choose $[z_n] \in [W]$ with $\Vert [z_n]\Vert > 1 - \eps$.  Let $(y_j) = (z_{n_j})$
be a subsequence of $(z_n)$ such that $\lim \Vert y_n \Vert = \Vert [z_n] \Vert$ and
$(y_n)$ converges weakly to an element $y$ in $K$.  There is no loss in generality
in assuming that
$\Vert y_n \Vert > 1 - \eps$ for all $n\in\N$ and in
choosing $y_n^*\in X^*$ so that $\Vert y_n^*\Vert = 1$,
 $y_n^*(y_n) = \Vert y_n\Vert$, and $(y_n^*)$ converges weak$^*$ to $y^*$. (This
 is possible because
the \fpp is separably-determined \cite[page 35]{GK};  so there is no loss in generality
in assuming that $B_{X^*}$ is weak$^*$-sequentially compact.)

From the definition of $[W]$
and the weak lower semicontinuity of the norm, it follows that, if $n$ is large enough,
\[
\Vert y_n - y \Vert \leq \liminf_{m} \Vert y_n - y_m \Vert < \frac{1+ \eps}{2}  \qquad \text{and}\qquad
\Vert y \Vert \leq \liminf_{j} \Vert y_j - x_{n_j} \Vert< \frac{1+ \eps}{2}\,.
\]
Therefore, with $u_n =  \frac{2}{1+\eps}(y_n - y)$ and $u = \frac{2}{1+\eps}y$,
\[
 \Vert u_n + u \Vert = \left\Vert  \frac{2}{1+\eps}(y_n - y) +  \frac{2}{1+\eps}y \right\Vert =
   \frac{2}{1+\eps} \Vert y_n\Vert >  2\,\frac{1 - \eps}{1+\eps} >
2(1-2\eps)
\]
if $n\in\N$ is large enough.
Applying the weak lower semicontinuity of the norm again, it follows that
\[
\liminf_m \Vert (u_n - u_m) + u \Vert \geq \Vert u_n + u\Vert > 2(1 - 2\eps)
\]
if $n\in\N$ is large enough.  So, by taking another subsequence if necessary, we can
assume that   $\Vert u_n + u\Vert > 2(1-2\eps)$ and $\Vert (u_n - u_m) + u \Vert  > 2(1 - 2\eps)$ for
all $n$ and all $m>n$.

Furthermore, since $y_m^*\overset{\text{w$^*$}}{\to} y^*$,

\[
\begin{aligned}
\liminf_m \Vert (u_n - u_{m}) - u \Vert &= \liminf_m \Vert (u_m + u) - u_n\Vert \\
   &\geq \liminf_m y_m^*\big( (u_m + u) - u_n\big) \\
   &= \liminf_m \big( \Vert u_m + u\Vert - y_m^*(u_n) \big) \\
   &\geq 2(1- 2\eps) - y^*(u_n)\,.
   \end{aligned}
\]
Then, since $u_n\overset{\text{w}}{\to} 0$,  it follows that
$\liminf_m \Vert (u_n - u_{m}) - u \Vert > 2(1 - 3\eps)$
if $n$ is large enough.
Therefore, for $n$ large enough and  $m>n$ also large enough,  both
\[
\Vert (u_n - u_m) + u \Vert  > 2(1 - 3\eps)\quad \text{and}\quad \Vert (u_n - u_{m}) - u \Vert > 2(1 - 3\eps)
\]
hold.
Since $\eps>0$ is arbitrary, $\Vert u_n - u_m\Vert < 1$, and $\Vert u\Vert < 1$, the above
inequalities imply that
$X$ fails to be uniformly nonsquare, a contradiction which finishes the proof.

\end{proof}

We want to refine the sequences $(x_{n_j})$, $(y_j)$, and $(y_j^*)$
that appear in the proof of Theorem \ref{T:uns}.   Recall the result of Goebel
and Karlovitz  \cite[page 124]{GK}:  If $K$ is a minimal $T$-invariant, weakly
compact, convex set for the nonexpansive map $T$ and $(x_n)$ is an approximate
fixed point sequence for $T$ in $K$, then the sequence $(\| x_n -x \|)$ converges
to the diameter of $K$ for every $x$ in $K$.

Fix $\eps >0$ and set $\x_1 = x_{n_1}$,  $\y_1=y_1$, and $\y_1^{\,*} = y_1^*$.
Then, by the Goebel-Karlovitz Lemma,  there exists $j_1>1$ such that
$\min \{ \| \x_1 - x_{n_{j_1}} \|, \| \y_1 - x_{n_{j_1}} \| \} > 1 -\eps$.  Set
$\x_2 = x_{n_{j_1}}$, $\y_2 = y_{j_1}$, and  $\y_2^{\,*} = y_{j_1}^*$.
 Another application of the
Goebel-Karlovitz Lemma yields $j_2 > j_1$ such that
$\displaystyle\min_{i = 1,2} \{ \| \x_i - x_{n_{j_2}} \|, \| \y_i - x_{n_{j_2}} \| \}> 1 -\eps$.
Set  $\x_3 = x_{n_{j_2}}$, $\y_3 = y_{j_2}$, and $\y_3^{\,*} = y_{j_2}^*$.
Continuing in this manner, we obtain sequences $(\x_n)$ and $(\y_n)$ in $K$
(and $B_X$) and
a sequence $(\y_n^{\,*})$ in $S_{X^*}$ satisfying
$\displaystyle\min_{i < n } \{ \| \x_i - \x_n \|, \| \y_i - \x_n \| \}> 1 -\eps$
for all $n\in\N$ and $\y_n^{\,*}(\y_n) = \| \y_n\| > 1 - \eps$ for all $n\in \N$\,.  In
the following, these sequences are renamed by omitting the tildes.  The following
result is a summary of several easy computations.

\begin{lemma} \label{L:numbers}
Let $X$ be a Banach space
whose dual unit ball is weak$^{\,*}$ sequentially compact
and assume that $X$ fails
 the weak fixed point property.  Given $\eps>0$, there exist sequences
$(y_n)$ in $B_X$ and $(y_n^*)$ in $S_{X^*}$ and elements $y\in B_X$ and $y^*\in B_{X^*}$
satisfying:
  \begin{enumerate}
  	\item $y_n\overset{\text{w}}{\to} y$ and $y_n^*\overset{\text{w$^*$}}{\to} y^*$;  \label{w}
	\item For every $n\in\N$, $1 - \eps < \| y_n\|  = y_n^* (y_n) \leq 1$;  \label{n-n}
	\item For every $n\in\N$, $ \frac{1 - 3\eps}{2} <  y_n^* (y) \leq \|y\| < \frac{1+\eps}{2}$;  \label{y}
	\item $\frac{1-3\eps}{2} < \|y_n - y\| < \frac{1+\eps}{2}$; \label{y-n}
	\item If $n\ne m$, then $\frac{1-3\eps}{2} < \|y_n - y_m\| < \frac{1+\eps}{2}$; \label{n-m}
	\item If $n\ne m$, then $\frac{1-3\eps}{2} < y_n^*( y_m) < \frac{1+2\eps}{2}$; \label{n*-m}
	\item $\frac{1-3\eps}{2} \leq y^*( y) \leq \frac{1+\eps}{2}$; \label{y*y}
	
  \end{enumerate}
\end{lemma}

\begin{proof}
 Claims \eqref{w} and \eqref{n-n}, the third inequality in  \eqref{y}, and the
 second inequalities in
  \eqref{y-n} and \eqref{n-m}  are immediate from the proof of Theorem \ref{T:uns}.
 Then
 \[
\|y\| \geq  y_n^*(y) = y_n^*(y_n) - y_n^*(y_n - y) > (1-\eps) - \frac{1+\eps}{2} = \frac{1-3\eps}{2}
 \]
 proving \eqref{y}.

 Also
 \[
 \| y_n -y\| \geq y_n^*(y_n - y) = \| y_n\| - y_n^*(y) > (1-\eps)  -\frac{1+\eps}{2}
 = \frac{1-3\eps}{2}
 \]
 which finishes the proof of  \eqref{y-n}.

 From our refinement of $(x_n)$ and $(y_n)$
 done just prior to the lemma and the definition of $[W]$ in the proof of Theorem \ref{T:uns},
 if $n > m$,
 \[
 \begin{aligned}
 \| y_n - y_m\| &= \| (y_n - x_n) + (x_n - y_m)\| \\
  &\geq  \|x_n - y_m\| - \| y_n - x_n\| \\
   &> (1-\eps) - \frac{1+\eps}{2} \\
   &= \frac{1 - 3\eps}{2}
 \end{aligned}
\]
showing that \eqref{n-m} holds.

The lower inequality in \eqref{n*-m} follows from \eqref{n-n} and \eqref{n-m}:
If $n\ne m$,
\[
1-\eps < y_n^*(y_n) = y_n^*(y_n - y_m) + y_n^*(y_m) \leq \|y_n - y_m\| + y_n^*(y_m)
< \frac{1+\eps}{2} + y_n^*(y_m)\,.
\]
Therefore, if $n\ne m$, $\frac{1-3\eps}{2} <  y_n^*( y_m)$.

In order to show the upper inequality in \eqref{n*-m}, we consider subsequences of
the sequences $(y_n)$ and $(y_n^*)$ obtained so far.  Note that all of the previous
conditions  will remain true  for subsequences of the current $(y_n)$ and $(y_n^*)$.
Let $\y_1 = y_1$ and $\y_1^{\,*} = y_1^*$.
Since $(y_n)$ converges weakly to $y$ and
$ \frac{1 - 3\eps}{2} <  y_1^* (y) < \frac{1+\eps}{2}$, there exists $n_1\in\N$ such
that, if $n\geq n_1$, $ \frac{1 - 3\eps}{2} <  y_1^* (y_{n}) < \frac{1+\eps}{2}$.
Set $\y_2 = y_{n_1}$ and $\y_2^{\,*} = y_{n_1}^*$.
Since $(y_n)$ converges weakly to $y$ and
$ \frac{1 - 3\eps}{2} <  y_{n_1}^* (y) < \frac{1+\eps}{2}$, there exists $n_2\in\N$ such
that, if $n\geq n_2$, $ \frac{1 - 3\eps}{2} <  y_{n_1}^* (y_{n}) < \frac{1+\eps}{2}$.
Set $\y_3 = y_{n_2}$ and $\y_3^{\,*} = y_{n_2}^*$.  Continuing in this manner
generates sequences $(\y_n)$ and $(\y_n^{\,*})$ satisfying conditions (1)--(5) and
satisfying $\y_n^{\,*}( \y_m) < \frac{1+\eps}{2}$
 if $n < m$.  Again, we simplify the notation by considering these new sequences
 but omitting the tildes in the notation.

To show the upper inequality in \eqref{n*-m} for $n>m$, first combine \eqref{y} with the
weak$^*$ convergence of $(y_n^*)$ to $y^*$ to obtain \eqref{y*y}.  Then,
since $(y_n)$ converges weakly to $y$, there is no loss in generality in
assuming that $y^*(y_n) < \frac{1 + 2\eps}{2}$ for all $n\in\N$.  In particular,
$y^*(y_1) < \frac{1 + 2\eps}{2}$.  Therefore, since $(y_n^*)$ converges
weak$^*$ to $y^*$, there exists $n_1\in\N$ such that, if $n\geq n_1$,
$y_{n}^{\,*}(y_1) < \frac{1 + 2\eps}{2}$.  Setting $\y_1 = y_1$, $\y_1^{\,*}=y_1^*$,
$\y_2 = y_{n_1}$, and $\y_2^{\,*} = y_{n_1}^*$ gives
$\y_2^{\,*}(\y_1) < \frac{1 + 2\eps}{2}$.  Then, since $y^*(\y_2) < \frac{1 + 2\eps}{2}$
and $(y_n^*)$ converges weak$^*$ to $y^*$, there exists $n_2\in\N$ such that $n_2> n_1$,
and, if $n\geq n_2$,  $y_{n}^{\,*}(\y_2) < \frac{1 + 2\eps}{2}$.  Set $\y_3 = y_{n_2}$ and
$\y_3^{\,*} = y_{n_2}^*$.  Continuing in this manner generates sequences $(\y_n)$
and $(\y_n^{\,*})$ satisfying all of the conditions of the lemma.

\end{proof}

As a consequence of these computations, we have the following

\begin{theorem}  \label{T:w*ukk}
Let $X$ be a Banach space such that $B_{X^*}$ is weak$^*$ sequentially compact.
If $X^*$ has the weak$^*$ uniform Kadec-Klee property, then $X$ has the weak fixed
point property for nonexpansive mappings.
\end{theorem}

\begin{proof}
If $X$ fails to have the weak fixed point property, consider the sequences
$(y_n)$ and $(y_n^*)$ determined in Lemma \ref{L:numbers}.  In particular,
note that $\|y_n^*\| = 1$ for all $n\in\N$ and that $(y_n^*)$ converges
weak$^*$ to $y^*$.  Note also that, if $n\ne m$,
\[
\begin{aligned}
(y_n^* - y_m^*)(y_n - y_m) &= y_n^*(y_n) - y_n^*(y_m) -  y_m^*(y_n) +  y_m^*(y_m) \\
  &> (1-\eps) - \frac{1+2\eps}{2} - \frac{1+2\eps}{2} + (1-\eps) \\
  &= 1 - 4\eps
\end{aligned}
\]
It follows that
\[
2 \geq \|y_n^* - y_m^*\| \geq (y_n^* - y_m^*)\left(\frac{y_n - y_m}{\|y_n - y_m\|}\right)
  > \frac{1 - 4\eps}{(1+\eps)/2} = 2\,\frac{1 - 4\eps}{1+\eps} > 2 - 10\eps\ .
\]
Thus, if $\eps <\frac{1}{10}$,  $(y_n^*)$ is a sequence in the unit sphere of $X^*$,
$(y_n^*)$ converges weak$^*$ to $y^*$, and $\text{sep}\{y_n^*\} > 1$.  Therefore,
by the weak$^*$ uniform Kadec-Klee property of $X^*$, there exists $\delta>0$
such that
\[
\|y^*\| < 1 - \delta\,.  \tag{*}
\]

But, by \eqref{y} and \eqref{y*y},
\[
\frac{1-3\eps}{2} \leq y^*(y) \leq \|y\|\, \|y^*\| < \frac{1+\eps}{2}\, \|y^*\|\,.
\]
Therefore, if $\eps < \min\left\{\frac{1}{10}, \frac{\delta}{4}\right\}$
\[
\|y^*\| > \frac{1-3\eps}{1+\eps} > 1 - 4\eps > 1 - \delta\, ,
\]
a contradiction to (*).  Therefore $X$ has the weak \fpp for nonexpansive mappings.
\end{proof}

Of course the theorem implies that $c_0$ with its usual norm has the weak \fpp
which was a result first proven by Maurey \cite{M}.  Since $H^1$ has the weak$^*$
uniform Kadec-Klee property \cite{BDDL},
its predual $C(T)/A_0$ has the weak \fpp by this theorem.
In the same manner, since ${\mathcal C}\sb 1(H)$,  the ideal of trace class
operators on a Hilbert space $H$, has the weak$^*$ uniform Kadec-Klee
property \cite{Le}, its predual ${\mathcal C}\sb \infty(H)$, the ideal of
compact operators in $B(H)$,
 has the weak fixed point property.
Since quotients of Banach spaces with weak$^*$ sequentially compact dual
unit balls  have weak$^*$ sequentially compact dual
unit balls \cite[page~227]{D}, it is easy to check that the following holds.

\begin{corollary}  \label{T:quotient}
Let $X$ be a Banach space such that $B_{X^*}$ is weak$^*$ sequentially compact.
If $X^*$ has the weak$^*$ uniform Kadec-Klee property and $Y$ is a closed subspace
of $X$, then $X/Y$ has the weak fixed point property for nonexpansive mappings.
\end{corollary}

We note that the corollary implies that the quotients of $c_0$ have the
weak fixed point property.  This is implicit in the work of Borwein and Sims \cite{BS}.

The authors had hoped that the sequences identified in Lemma
\ref{L:numbers} would be useful
in establishing connections between superreflexive Banach spaces
and the \fpp for nonexpansive mappings.
Consider the  sequences generated in Lemma \ref{L:numbers}
for each $\eps = \frac1k$, $k\in \N$.  That is,
 for each $k\in\N$, let $(y_{k,n})$ and $(y_{k,n}^*)$ denote the sequences
$(y_n)$ and $(y_n^*)$ constructed in Lemma \ref{L:numbers} with
$\eps = \frac1k$; let $y_{k,\infty}$ denote the weak limit of the
sequence $(y_{k,n})$; and let $y_{k,\infty}^*$ denote the weak$^*$ limit of the
sequence $(y_{k,n}^*)$.  For a non-trivial ultrafilter $\U$ on $\N$, let
$X_\U$ denote the ultrapower of $X$ with respect to $\U$.  (For
information on ultraproducts in Banach space theory, see
\cite{H} or \cite{S}.)  Define sequences $(\bfy_n)$ in $X_\U$,
$(\bfy_n^*)$ in $(X^*)_\U$, and the point $\bfy$ in
$X_\U$ by
\[
\bfy_n = (y_{1,n},\, y_{2,n},\, y_{3,n}, \cdots)_\U\, ,
\]
\[
\bfy_n^* = (y_{1,n}^*,\, y_{2,n}^*,\, y_{3,n}^*, \cdots)_\U\, ,\ \text{and}
\]
\[
\bfy = (y_{1,\infty},\, y_{2,\infty},\, y_{3,\infty}, \cdots)_\U\,.
\]
The pair of sequences  $\big(2(\bfy_n - \bfy)\big)$ and $(\bfy_n^*)$ forms
a biorthogonal system of norm-one elements in $X_\U$ and $(X^*)_\U$ and,
for each sequence $(\alpha_n)$ of nonnegative real numbers,
$\| \sum_{n=1}^\infty \alpha_n \bfy_n^*\| = \sum_{n=1}^\infty \alpha_n$.  Moreover,
as is clear from the proof of Theorem \ref{T:w*ukk},  $\|\bfy_m^* - \bfy_n^*\| = 2$ for all
$m\ne n$.   Initially, the authors felt that, if $X$ was a renorming of $\ell^2$,
 this ``positive $\ell^1$-type behavior" should not
occur in $(X^*)_\U $ since  $(X^*)_\U $ would also be a renorming of a Hilbert space.
However, as first pointed out to us by by Professor V.D.\ Milman, there do
exist renormings of $\ell^2$ with this behavior.   A second example
resulted from a discussion with Professors A.\ Pe{\l}czy{\'n}ski and
M.\ Wojciechowski.
In fact, every infinite-dimensional Banach space can be renormed to exhibit this
$\ell^1$-type behavior for nonnegative linear combinations.  To see this, let
$(x_i, x_i^*)$ in $X\times X^*$ be a biorthogonal system with $\|x_i\| = 1$
and $\|x_i^*\|\leq 2$.  (Such a biorthogonal system exists by applying  a theorem of
Ovsepian and  Pe{\l}czy{\'n}ski \cite[page 56]{D} to a separable subspace of
$X$ and then extending to functionals on all of $X$ via the Hahn-Banach theorem.)
Then
$|||x||| = \max\{ |x_1^*(x)|, \frac12 \|x\|, \displaystyle\sup_{i\ne j; \, i, j\geq 2} (\,|x_i^*(x)| + |x_j^*(x)|\, ) \}$ defines
an equivalent norm on $X$ with $||| x_1 + x_n|||=1$ and
$|||\, \sum_{n=1}^\infty \alpha_n (x_1 + x_n)\,||| =  \sum_{n=1}^\infty \alpha_n$ if $\alpha_n\geq 0$.
(For related examples, see Example 3.13 in \cite{NS} and
Theorem 7 in \cite{K}.)

Despite the above disappointment, the sequence $(\bfy_n^*)$ in $(X_\U)^*$ or the sequences
$(y_n^*)$ in $X^*$ for a given $\eps$ in Lemma \ref{L:numbers} can be used to generalize
Theorem \ref{T:uns}.  A subset $A$ of $X$ is {\it symmetrically $\eps$-separated} if
the distance between any two distinct points of $A \cup (-A)$ is at least $\eps$ and
a Banach space $X$ is {\it $O$-convex} if the unit ball $B_X$ contains no
symmetrically $(2-\eps)$-separated subset of cardinality $n$ for some $\eps>0$ and
some $n\in\N$ \cite{NS}.   $O$-convex spaces are superreflexive.  Therefore the proof
of Theorem \ref{T:w*ukk}  combines with property (3) in Lemma \ref{L:numbers} to
show that, if $X$ fails to have the fixed point property, then, for every $\eps>0$, there
exists a countably infinite set $A = \{y_1^*, y_2^*, \cdots \}$
 in the unit sphere of $X^*$ such that
$A\cup (-A)$ is $(2-\eps)$-separated.  In particular, this implies:

\begin{theorem} \label{T:O-conv}
If $X^*$ is $O$-convex, then the Banach space $X$ has the \fpp for
nonexpansive mappings.
\end{theorem}

Since uniformly nonsquare Banach spaces are $O$-convex, Theorem \ref{T:O-conv}
is a generalization of Theorem \ref{T:uns}.  Naidu and Sastry \cite{NS} also characterized
the dual property to $O$-convexity.  For $\eps>0$, a convex subset $A$ of $B_X$
is an {\it $\eps$-flat} if $A \cap (1-\eps) B_X = \emptyset$.
Note that the convex hulls of the sets $\{y_1^*, y_2^*, \cdots \}$ from Lemma \ref{L:numbers}
are $3\eps$-flats.
A collection $\mathcal{D}$ of
$\eps$-flats is {\it jointly complemented} if, for each distinct $\eps$-flats $A$ and $B$ in
$\mathcal{D}$, the sets $A\cap B$ and $A\cap (-B)$ are nonempty.
Define
\[
E(n,X) = \inf \{\eps: B_X \text{ contains a jointly complemented collection of
$\eps$-flats of cardinality $n$} \}.
\]
A Banach space $X$ is {\it $E$-convex} if $E(n,X)>0$ for some $n\in\N$.
In \cite{SS}, S.\ Saejung noted that $E$-convex spaces may fail to have normal
structure and asked if $E$ convex spaces have the fixed point property.  Since a
Banach space is $E$-convex if and only if its dual space is $O$-convex,
Theorem \ref{T:O-conv} can be restated to give a positive answer to Saejung's
question.

\begin{theorem} \label{T:E-conv}
$E$-convex spaces have the \fpp for
nonexpansive mappings.
\end{theorem}

For a detailed
analysis of $O$-convex, $E$-convex, and related properties in the hierarchy between
Hilbert spaces and reflexive spaces, see \cite{AF, NS, SS}.
\bigskip

\flushleft
{\bf Acknowledgement.}  The authors thank  V.\ D.\ Milman, A.\ Pe{\l}czy{\'n}ski, N.~Randrianantoanina,
and M. Wojciechowski for helpful discussions during the preparation of this article.

\bibliographystyle{amsplain}

\end{document}